\documentclass[11pt]{article}

\usepackage{verbatim}
\usepackage{hyperref}
\usepackage{harpoon}
\usepackage{enumerate, amsthm}
\usepackage{mathrsfs}
\usepackage{rotating}
\usepackage{cancel}
\usepackage{enumitem}
\usepackage{xcolor}
\usepackage[margin=1.1 in]{geometry}
\usepackage[margin={1.5cm, 1.5cm},font=small, labelsep=endash]{caption}
\usepackage{mathtools}
\usepackage{tikz}
\usetikzlibrary{patterns}
\usepackage{ifthen}

\usepackage{amssymb}
\usepackage{amsmath}
\usepackage{centernot}

\numberwithin{equation}{section}
\theoremstyle{plain}
\newtheorem{theorem}{Theorem}[section]
\newtheorem{lemma}[theorem]{Lemma}

\newtheorem{proposition}[theorem]{Proposition}

\newtheorem{problem}[theorem]{Problem}

\theoremstyle{definition}

\theoremstyle{remark}
\newtheorem{remark}[theorem]{Remark}

\newcommand{\R}{{\mathbb R}}

\renewcommand{\P}{{\bf P}}
\newcommand{\Pb}{{\mathbb P}}
\newcommand{\Eb}{{\mathbb E}}

\newcommand{\Z}{\mathbb{Z}}

\renewcommand{\L}{\Lambda}

\newcommand{\lr}[4]{#3\mathrel{\mathop{\longleftrightarrow}_{#1}^{#2}} #4}

\newcommand{\nlr}[4]{#3\mathrel{\mathop{\centernot\longleftrightarrow}_{#1}^{#2}} #4}

\makeatletter
\def\@setcopyright{}
\def\serieslogo@{}
\makeatother

\begin{document}
	\title{Existence of phase transition for percolation using the Gaussian Free Field}
	
	\author{
		Hugo Duminil-Copin\footnotemark[1]\footnote{Universit\'e de Gen\`eve} \footnotemark[2]\footnote{Institut des Hautes \'Etudes Scientifiques} , Subhajit Goswami\footnote{School of Mathematics, Tata Institute of Fundamental Research} , Aran Raoufi\footnotemark[2] , Franco Severo\footnotemark[1] \footnotemark[2] , Ariel Yadin\footnote{Ben-Gurion University of the Negev}}
	
	\date{\today}
	
	\maketitle
	
	\begin{abstract}
		In this paper, we prove that Bernoulli percolation on bounded degree graphs with isoperimetric dimension $d>4$ undergoes a non-trivial phase transition (in the sense that $p_c<1$). As a corollary, we obtain that the critical point of Bernoulli percolation on infinite quasi-transitive graphs (in particular, Cayley graphs) with super-linear growth is strictly smaller than 1, thus answering a conjecture of Benjamini and Schramm. The proof relies on a new technique consisting in expressing certain functionals of the Gaussian Free Field (GFF) in terms of connectivity probabilities for a percolation model in a random environment. Then, we integrate out the randomness in the edge-parameters using a multi-scale decomposition of the GFF. We believe that a similar strategy could lead to proofs of the existence of a phase transition for various other models.
	\end{abstract}
	
	\section{Introduction} \label{sec:intro}
	
	\paragraph{Motivation.}
	Whether a model undergoes a non-trivial phase transition or not is one of the most fundamental questions in statistical physics. In \cite{Pei36}, Peierls introduced a combinatorial technique, known as {\em Peierls argument}, to prove that the critical temperature of the Ising model is non-zero on $\mathbb Z^d$ for $d\ge2$, thus opening a new realm of questions concerning this important model of ferromagnetism. This argument found many applications to other models, including Potts models and models of random graphs such as Bernoulli percolation or the random-cluster model. 
	
	Peierls argument has two drawbacks. First, it often does not apply to continuous spin models, for instance the spin $O(n)$ models. In this case, the technique may sometimes be replaced by two other techniques: {\em Reflection Positivity} and the {\em Renormalization Group}. More precisely, Fr\"ohlich, Simon and Spencer proved that the spin $O(n)$ model undergoes a non-trivial order/disorder phase transition on $\mathbb Z^d$ with $d\ge3$ \cite{FSS76} using Reflection Positivity, and Balaban and coauthors (see \cite{Bal} and references therein) proved delicate properties of the large $\beta$ regime using the Renormalization Group. Let us mention that in the special case of the XY model (i.e.~when $n=2$), there are special proofs relying on the Coulomb gas \cite{FS82,KK86}.
	
	Another problem with Peierls argument is that it requires a precise understanding of so-called {\em cut sets}, i.e.~sets of edges which disconnect certain sets of vertices from infinity. On planar graphs, this boils down to the understanding of circuits in the dual graph. On non-planar graphs, the question is a much more complex combinatorial problem and it is not completely understood in general.
	
	In this paper, we wish to present a new technique which we believe can be useful to prove the existence of a phase transition for various models. The object of interest of this paper will be Bernoulli percolation.
	
	\paragraph{Main results.} Consider a connected graph $G=(V,E)$ with vertex-set $V$ and edge-set $E$.  An edge with endpoints $x$ and $y$ will be denoted by $xy$. For every $x\in V$, let $d(x)$ be the number of $y$ such that $xy\in E$. 
	We will assume that the graph has {\em bounded degree}, that is $\sup\{d(x):x\in V\}<+\infty$.
	
	{\em Bernoulli percolation} is a model of a random subgraph of $G$ with vertex-set $V$. The subgraph is encoded by a function $\omega$ from $E$ into $\{0,1\}$. We use the notation $\omega_{xy}$ to denote the value of $\omega$ at the edge $xy$ and think of edges $xy$ with $\omega_{xy}=1$ as being the edges of the subgraph. These edges are called {\em open}, while those with $\omega_{xy}=0$ are called {\em closed}. We are interested in the connectivity properties of the graph $\omega$. We use the notation $\lr{}{}{S}{T}$ (resp.~$\lr{}{}{S}{\infty}$) 
	to denote the event that 
	there is a nearest neighbor open path in $\omega$ connecting a vertex in $S$ to a vertex in $T$ 
	(resp. the event that $S$ intersects an infinite connected component of $\omega$). 
	Also, let $\nlr{}{}{S}{T}$ (resp. $\nlr{}{}{S}{\infty}$) denote the complement of the event $\lr{}{}{S}{T}$ (resp. $\lr{}{}{S}{\infty}$).

	The Bernoulli bond percolation of parameter ${\bf p}=(p_{xy}:xy\in E)\in[0,1]^E$ is the probability measure on configurations $\omega$ for which the variables $\omega_{xy}$ form a family of independent Bernoulli random variables of respective parameters $p_{xy}$. We denote such a measure by ${\bf P}_{\bf p}$ and, when $p_{xy}=p$ for every $xy\in E$, we simply write ${\bf P}_p$. The main question of interest is whether the critical parameter 
	$$p_c(G):=\inf\big\{p\in[0,1]:{\bf P}_p[x\longleftrightarrow\infty]>0\big\}$$(here $x\in V$ is a vertex chosen arbitrarily)
	is strictly smaller than 1 or not. Let us mention that proving $p_c(G)>0$ is a much simpler task: a simple comparison with branching process actually implies $p_c(G)\ge\frac{1}{D}$, where $D$ is the maximum degree of $G$.

	In order to state the main result, we need another notion. Given a graph $ G$, {\em simple random walk} is the discrete-time Markov chain $(X_n)_{n\geq0}$ on $V$ moving, at each time step, from its position to one of its neighbors in $G$ chosen uniformly at random. Define the {\em heat kernel} and the {\em Green function} by the following formula: for every $x,y\in V$ and $n\ge0$,
	$$p_n(x,y):={\rm P}[X_n=y \vert X_0=x]\quad\text{ and }\quad \mathsf{G}(x,y):=\frac{1}{d(y)}\sum_{n=0}^{\infty} p_n(x,y).$$
	
	The main object of this paper is the following result.
	\begin{theorem}
		\label{thm:main}
		Consider a graph $G$ with bounded degree. Assume that there exist real numbers $d>4$ and $c>0$ such that	\begin{equation}
		\label{eq:heat.kernel.bound}\tag{H$_d$}
		p_n(x,x)\leq \frac{c}{n^{d/2}}\qquad\forall x\in V\ ,\ \forall n\ge1.
		\end{equation}
		Then, there exists 
		$p<1$ such that for every finite set $S\subset V$,
		\begin{equation}
		\label{eq.main.result}
		\P_{p}[\nlr{}{}{S}{\infty}]\le \exp[-\tfrac12\mathrm{cap}(S)]		,
		\end{equation}
		where 
		$\mathrm{cap}(S):=\sum_{x\in S} d(x){\rm P}[X_k\notin S,\forall k\ge1 \vert X_0=x]$ is the capacity of $S$. In particular, since $\mathrm{cap}(S)>0$ iff $G$ is transient (which is implied by \eqref{eq:heat.kernel.bound}, $d>2$), one has $p_c( G)<1$.
	\end{theorem}
	
	Let us mention a few applications of the theorem. We say that a graph $G$ satisfies an {\em isoperimetric inequality of dimension $d$} if there exists some constant $c>0$ such that
	\begin{equation}\label{eq:isoper}\tag{I$_d$}
	|\partial K|\ge c|K|^{\frac{d-1}{d}}~~\text{for all finite}~K\subset V,
	\end{equation}
	where $\partial K := \{\{x,y\}\in E : x\in K,~ y\notin K \}$ is the edge boundary of $K$. The {\em isoperimetric dimension} of $G$, which we denote by ${\rm Dim}(G)$, is defined as the supremum over all $d$ such that \eqref{eq:isoper} holds. In their celebrated paper \cite{BenSch96}, Benjamini and Schramm asked whether ${\rm Dim}(G)>1$ necessarily implies $p_c( G)<1$. We give the following partial answer to this question:
	
	\begin{theorem}
		\label{thm:isoper}
		If a bounded-degree graph $G$ satisfies ${\rm Dim}(G)>4$, then $p_c( G)<1$. 
	\end{theorem}
	This theorem follows directly from Theorem \ref{thm:main} and a result of Varopoulos \cite{Var85} (see also \cite{MorPer05} or \cite[Corollary~6.32]{LyoPer17} for a proof relying on evolving sets) stating that \eqref{eq:isoper} implies \eqref{eq:heat.kernel.bound}.

	An important application of Theorem~\ref{thm:isoper} is the following result answering the first two conjectures of Benjamini and Schramm from \cite{BenSch96}. The graph $G$ is called {\em quasi-transitive}
	if the action of the automorphism group ${\rm Aut}(G)$ on $V$ has only finitely many orbits. Typical examples of quasi-transitive graphs to keep in mind are Cayley graphs of finitely generated groups.
	Let ${\rm B}_r(x)$ be the ball of radius $r$ centered at $x$ with respect to the graph distance. We say that a graph $G$ has {\em super-linear growth} if $\limsup \tfrac1r|{\rm B}_r(x)|=+\infty$. 			\begin{theorem}\label{thm:main2}
		Let $ G$ be a quasi-transitive graph with super-linear growth, then $p_c( G)<1$.
	\end{theorem}
	
	The previous result can be deduced from Theorem~\ref{thm:isoper} as follows:
	\begin{itemize}
		\item If $\liminf \frac{1}{r^{d}}|{\rm B}_r(x)|<+\infty$ for some $d>0$, then a result of Trofimov \cite{Tro84} (see also \cite[Theorem 5.11]{woess2000random}) implies that the graph is quasi-isometric to a Cayley graph with polynomial growth. This fact together with super-linear growth classically implies that $p_c(G)<1$ (see the next section for details). 
		\item If $\liminf \frac{1}{r^{d}}|{\rm B}_r(x)|=+\infty$ for every $d>0$, then in particular $\liminf \frac{1}{r^{d}}|{\rm B}_r(x)|>0$ for some $d>4$. Therefore \cite[Lemma 7.2]{LyoMorSch08} (see also the proof of Corollary 7.3 of the same paper or \cite{CSC93} for the special case of Cayley graphs) implies that the graph satisfies \eqref{eq:isoper}, which in turn gives that $p_c(G)<1$ by Theorem \ref{thm:isoper}.
	\end{itemize}
	
	All the results in this paper can be extended to site percolation, finite dependent percolation and random-cluster models via classical comparisons, see  respectively \cite{GriSta98}, \cite{LigSchSta97} and \cite{Gri06}. The coupling between random-cluster models and the Ising/Potts model implies that the results translate into results on the latter as well.
	
	\paragraph{Existing results.} 
	Our results should be put in context with the previous partial results regarding the Benjamini-Schramm questions.

	As already mentioned, Peierls proved \cite{Pei36} that the Ising phase transition is non-trivial for $\Z^2$ through bounding the number of cut sets of specific size disconnecting a vertex from infinity. Peierls' proof also applies to Bernoulli percolation and can be easily extended to all $\Z^d$ ($d\geq2$) using monotonicity arguments. See also Lebowitz and Mazel \cite{lebowitz1998improved} and  Balister and Bollob\'as \cite{balister2007counting} for estimates on the number of cut sets of $\Z^d$.
	The cut set method was extended to Cayley graphs of finitely presented groups by Babson and Benjamini \cite{babson1999cut} (see also the work of Timar \cite{timar2007cutsets}). As of today, a  technique bounding the number of cut sets of a certain cardinality has not been found for general graphs.
		
	Say that a graph $G$ has {\em polynomial growth} if $\limsup \tfrac1{r^d}|{\rm B}_r(x)|<+\infty$ for some $d>0$. 		As a consequence of Gromov's celebrated theorem \cite{gromov1981groups}, every infinite finitely generated group of polynomial growth, which is not virtually $\Z$ (in the group theoretical sense), contains a subgroup isomorphic to $\Z^2$ \cite[Proposition 7.18]{LyoPer17}. Hence, there exists a Cayley graph of such groups that has a subgraph isomorphic to $\Z^2$. Since the property that $p_c(G)<1$ is stable under quasi-isometries  \cite[Theorem 7.15]{LyoPer17}, all the Cayley graphs of such groups have non-trivial phase transitions. This method was also used by Muchnik and Pak in \cite{MP01} to prove $p_c(G)<1$ for Grigorchuk groups which are a class of groups with intermediate (i.e.~faster than polynomial but slower than exponential) growth.
	
	Lyons has proved \cite{Lyons95} that every Cayley graph of {\em exponential growth} (i.e.~satisfying that $\liminf \tfrac1r\log |{\rm B}_r(x)|>0$) has a non-trivial phase transition. As noted in \cite[Page 264]{LyoPer17}, the fact that $p_c(G)<1$ for quasi-transitive graphs $G$ with exponential growth can also be easily obtained from the finiteness of the susceptibility for subcritical percolation; see \cite{AizBar87,Men86,DumTas15}.
	
	In \cite{BPP} Benjamini, Pemantle and Peres proved another criterion for 
	$p_c(G)<1$: the existence of an EIT measure for the graph. 
	A measure on self-avoiding paths starting from a fixed vertex is an EIT measure if the number of 
	intersections of two independent paths sampled according to the measure has an
	exponential tail. 
	In \cite{raoufi2017indicable}, by constructing an EIT measure, 
	it is proved that if $G$ is a Cayley graph of a virtually indicable group which is 
	not virtually $\Z$, then $p_c(G)<1$. Virtually indicable groups not only contain 
	groups of polynomial growth, they also include some groups of intermediate 
	growth.
	It is worth mentioning that the EIT method proves that for $p$ sufficiently close to $1$, there exists a transient infinite connected component almost surely. Unfortunately, it is usually difficult to construct EIT measures on general graphs.
	
	The question of $p_c(G)<1$ has also been approached by analyzing isoperimetric inequalities. Benjamini and Schramm proved in \cite{BenSch96} that if $G$ satisfies the isoperimetric inequality of ``dimension $\infty$'', i.e. if $G$ is non-amenable, then $p_c(G)<1$. It was proved in \cite{benjamini1999group} that every unimodular transitive non-amenable graph $G$ has a  threshold $\alpha<1$ such that any (not necessarily i.i.d.) automorphism invariant percolation measure on $G$ with density higher than $\alpha$ has an infinite connected component with positive probability. 
	Kozma proved in \cite{kozma2005percolation} that planar graphs of polynomial growth with no vertex accumulation points and isoperimetric dimension greater than $1$ have non-trivial phase transition. 
	
	In \cite{teixeira2016percolation}, Teixeira showed that $p_c(G)<1$ for graphs $G$ with polynomial growth and isoperimetric dimension greater than $1$ for a stronger version of the isoperimetric inequality, called \textit{local isoperimetric inequality}. Teixeira's proof relies on a clever renormalization argument using in a crucial way the (essential) uniqueness of large connected components in a box. It is important to note that this property is not satisfied under the sole assumption that ${\rm Dim}(G)>1$, as exemplified by the graph made of two copies of $\Z^d$ connected to each other by a single edge between two of their vertices. Also, in contrast to Teixeira's proof, our method {\em does not rely on uniqueness}: it works for graphs on which there may be any number of infinite connected components. The method of \cite{teixeira2016percolation} was further exploited in \cite{candellero2015percolation} to prove, without invoking Gromov's theorem, that $p_c(G) < 1$ for quasi-transitive graphs $G$ with super-linear but polynomial growth.	
	
	Let us conclude by pointing out that for a given \textit{explicit} graph $G$, it is often not hard to find some particular structure inside it that allows one to verify that $p_c(G)<1$. For instance, all known examples of Cayley graphs can be proved to have a phase transition without using the previous result. Nevertheless, groups not in the above known categories are discovered from time to time, see e.g.~\cite{Nek}, and without our result they need a case by case analysis.
	This should be compared to many group-theoretical properties of Cayley graphs that can often be proved (or disproved) for every explicit groups, yet are tremendously difficult to verify for the whole family of graphs under study. 
			
	\paragraph{Discussion on the strategy of proof.}
	The proof of Theorem \ref{thm:main} relies on a new connection between the Gaussian Free Field (GFF) and Bernoulli percolation. The connection goes through the observation that conditionally on the absolute value of the GFF at every point, the distribution of the signs of the GFF is the one of an Ising model with certain coupling constants. 
	  Once the connection between the GFF and the Ising is made, we use the Edwards-Sokal coupling to relate the Ising model to Bernoulli percolation. 
	As a result, it is possible to express the expectations  of particular observables of the GFF in terms of the probabilities of connections for a (annealed) percolation model on a random environment (i.e.~random edge-parameters). Since the expectation of these observables can be explicitly computed in terms of the Green function and are therefore easy to study, one may
	bound from below the averaged probability of connections in this percolation model.  One can also derive such results by exploiting connections from \cite{Lup16} and \cite{Szn12} between GFF on the metric graph $\tilde{G}$ and random interlacements, see remarks in the end of Section~\ref{sec:perco.on.gff} for details.
Let us mention that in recent years, GFF on the metric graph $\tilde{G}$ has efficiently explained a number of connections between GFF and percolation-type quantities, and that the previous statement is yet another illustration of the usefulness of this object.
	
	The second step in the proof of Theorem~\ref{thm:main} consists in integrating out the randomness of the environment in order to compare the probabilities of connection in the previous model to those in a Bernoulli percolation with fixed edge-parameter. Since the environment is 
	expressed in terms of the GFF, we will proceed step by step using a multi-scale decomposition of the GFF. More precisely, we will write $\varphi=\sum_n\varphi^n$ where the $\varphi^n$ are independent Gaussian fields with finite-range correlations. We will then remove the $\varphi^n$ one by one, while increasing an ``independent'' edge-parameter $q$ in order to guarantee that the probabilities of connection keep increasing. At the end of the process, the randomness due to the $\varphi^n$ (and therefore to $\varphi$) will have completely disappeared, and we will be facing a percolation model with constant edge-density.
	
	It is interesting to notice that we will not prove, in our second step, 
	that a percolation with some constant edge-parameter $p<1$ stochastically dominates the one on the random environment, because this would be simply false. Rather, we will only compare the probabilities of connections, which is a weaker statement.
	
	We believe that the present argument consisting in integrating out the long-range modes of the GFF will have further implications in the study of strongly correlated percolation models. For instance, sharpness of the phase transition of the GFF super-level set percolation is obtained in the forthcoming paper \cite{DGRS19}; the key step is a comparison between percolation probabilities for level-sets of the GFF and that of a truncated (finite-range) version of it, which is obtained by implementing a strategy similar to the one of this paper. 
	
	\paragraph{Open problems.}
	The present results raise a number of interesting problems. The first natural one is to try to relax the requirements on the heat kernel decay. More precisely, we will see that in the first step of the proof (Proposition \ref{prop:gff.percolation}) we only need the graph to be transient (which is true as soon as $d>2$), so that we can consider the GFF in infinite volume. The assumption that $d>4$ is used in the second step of the proof (Proposition \ref{prop:main}) for (we believe) purely technical reasons. This naturally raises the following problem.
	
	\begin{problem}
		Prove Theorem~\ref{thm:main} under the assumption that $d>2$.
	\end{problem}
	The main step in which we need $d>4$ is in the rewiring estimate of Step 3 in the proof of Lemma~\ref{lem:a2}, where a competition takes place between the exponential rewiring cost and a super-exponential gain coming from the assumption $d>4$. Replacing the exponential cost by a polynomial one in the rewiring estimate would enable one to prove the result for $d>2$. We believe that this problem is tractable in the case of quasi-transitive graphs. We chose not to present the proof since the result was already obtained by different means, but we wish to highlight the fact that the barrier $d=4$ is not related to intersection properties of simple random-walks.
		
	The proof uses the bounded degree assumption in one place only (in the last step of the proof of Lemma~\ref{lem:a2} again). It is therefore natural to ask the following problem.
	\begin{problem}
		Prove Theorem~\ref{thm:main} under the assumption that the graph is locally finite, meaning that $d(x)<\infty$ for every $x\in V$.
	\end{problem}
	Another natural problem is to improve \eqref{eq.main.result}. This bound is not sharp, even when applied in a simple context. Indeed, for $ G=\Z^d$ and $S$ a ball of radius $r$ one has $\mathrm{cap}(S)\asymp r^{d-2}$ (see (2.16), p. 53 of \cite{Law91}), therefore the upper bound provided by \eqref{eq.main.result} is of the form $\exp(-cr^{d-2})$ while the truth, for $p$  above $p_c(G)$, is rather $\exp(-cr^{d-1})$ (this can be easily derived from the main result of \cite{KZ90}).
	\begin{problem}
		Improve the bound \eqref{eq.main.result}.
	\end{problem}
	This problem is probably difficult with the current techniques, due to the following caveat. The percolation with random edge-densities introduced in this paper does not dominate any percolation model with a fixed positive edge-parameter. As a consequence, we believe that the probabilities of big but finite connected components do not have the same tail behavior as in standard Bernoulli percolation.
	
	The last problem is related to other models and is much more informal. In the next section, we will use that conditioned on the absolute value of the GFF, the signs of the GFF are sampled according to an Ising model. When conditioning the (Euclidean) norm of the $n$-component GFF, the normalized field is sampled according to a spin $O(n)$ model. As a consequence, the first step of our proof can be extended to this context and it is believable that the second step (comparing the model in random environment to a model with fixed coupling constants) could be adapted, even though the lack of correlation inequalities makes it a challenge. 
	\begin{problem}
		Use the techniques developed in the present paper to prove the existence of a phase transition for the spin $O(n)$ models. 
	\end{problem}
	
	\paragraph{Notation.} Set $u_+=\max\{u,0\}$ and ${\rm sgn}(u)=+1$ if $u\ge0$ and $-1$ otherwise. For a set $\L\subset V$, set  $\L^c:=V\setminus \L$. 
		
	\paragraph{Organization of the paper.} The next section presents the connection between the GFF and a percolation model with random edge-parameters. The third section implements the ``integration'' of the randomness in the edge-parameters.

	\section{GFF and Bernoulli percolation} \label{sec:perco.on.gff}

	In this section we consider $G=(V,E)$ to be any transient graph. Let $\Lambda$ be a finite subset of $V$. The {\em Gaussian Free Field} (or GFF) with $0$ boundary conditions on $\Lambda$ is defined to be the random (Gaussian) field $\varphi=(\varphi_x:x\in \Lambda)$ in $ \R^\Lambda$ with distribution  
	\begin{align*}
	d\Pb_{\Lambda}[\varphi] &:= \frac{1}{Z_{\Lambda}} \exp [-\mathcal D_{\Lambda}(\varphi)] d\varphi,\end{align*}
	where $Z_{\L}$ is a normalizing constant, $d\varphi$ stands for the Lebesgue measure on $\R^\Lambda$ and $\mathcal D_\Lambda(\varphi)$ is the {\em Dirichlet energy} given by
	$$\mathcal D_{\Lambda}(\varphi) := \tfrac{1}{2}\sum_{\substack{xy\in E\\ \{x,y\}\cap \L\ne \emptyset}} (\varphi_x-\varphi_y)^2,$$
	where $\varphi_x$ is extended to every vertex of $V$ by setting $\varphi_x=0$ for every $x\in \Lambda^c$. Under the assumption of transience of the graph $G$, which follows from \eqref{eq:heat.kernel.bound} for $d>2$, one can extend the GFF to $\Lambda=V$ by taking the weak limit $\Pb$ of the measures $\Pb_\L$ as $\L$ tends to $V$. The measure $\Pb$ is simply the centered Gaussian vector with covariance matrix given by the Green function $\mathsf{G}$, see \cite{Ber16}. Expectation with respect to $\Pb$ (resp. $\Pb_{\L}$) is denoted by $\Eb$ (resp. $\Eb_{\L}$).
	The main result of this section is the following.
	\begin{proposition}
		\label{prop:gff.percolation}
		Let $G$ be a transient graph. Then for any finite subset $S$ of $V$ one has
		\begin{equation} \label{eq:percolates}
		\Eb[\P_{{\bf p}(\varphi)}(\nlr{}{}{S}{\infty})] \leq \exp[-\tfrac12\mathrm{cap}(S)]		
		\end{equation}
		where
		${\bf p}(\varphi)_{xy}:= 1-\exp[-2(\varphi_x+1)_+\,(\varphi_y+1)_+]$ for every $xy\in E$.
	\end{proposition}
	
	Note that for $S=\{x\}$, one gets that $x$ is connected to infinity with positive annealed probability. One may wonder
	why we added 1 to the GFF: we refer to the remarks at the end of this section for a discussion of this technical trick.

	The key step in the proof of Proposition \ref{prop:gff.percolation} is the following lemma.
	
	\begin{lemma}
		\label{lem:representation}
		Let $G$ be a transient graph and fix a finite subset $S$ of $V$ and $t\in\R^{S}$. If $X_S^t(\varphi):=\exp[-\sum_{x\in S}t_x(\varphi_x+1)]$, then
		$$\Eb[\P_{{\bf p}(\varphi)}(\nlr{}{}{S}{\infty})]\le \Eb[X_S^t(\varphi)].$$
	\end{lemma}
	Before proving this lemma, let us show how it implies Proposition~\ref{prop:gff.percolation}.
	\begin{proof}[Proof of Proposition~\ref{prop:gff.percolation}]
		Since $\sum_{x\in S} t_x(\varphi_x+1)$ is a Gaussian random variable with mean $\sum_{x\in S} t_x$ and variance $\sum_{x,y\in S} t_x t_y \mathsf{G}(x,y)$, we deduce that 
		\begin{align*}
		\Eb[X_S^t(\varphi)]=\exp\Big(-\sum_{x\in S} t_x +\frac{1}{2}\sum_{x,y\in S} t_x t_y \mathsf{G}(x,y)\Big).
		\end{align*}
		Now, we choose $t$ according to the {\em equilibrium measure} of $S$, namely
		$$t_x=e_{S}(x):= d(x){\rm P}[X_k\notin S ~\forall k\ge1 \vert X_0=x]$$
		(which turns out to be the optimal choice of $t$). This gives the result by observing that $\mathrm{cap}(S)=\sum_{x\in S} e_{S}(x)$ and that $\sum_{y\in S} e_{S}(y) G(x,y)=1$ for all $x\in S$ (which can be deduced in a straightforward way via a decomposition of the random walk started at $x$ in terms of its last visit to $S$). 
	\end{proof}
	Let us now turn to the proof of the lemma. 
	
	\begin{proof}[Proof of Lemma \ref{lem:representation}]The proof proceeds in three steps. The first one relates the GFF on $\L$ to an Ising model on $\L$ with $+$ boundary conditions and random coupling constants. The second one relates the Ising model to Bernoulli percolation via the Edwards-Sokal coupling. The last step consists in taking the limit as $\L$ tends to $V$.
		
		In the first two steps, we fix a finite subset $\L$ of $V$ and consider $\mathbb P_\Lambda$.
		We also define 
		\begin{align}\label{eq:def_J}
		\begin{split}
		|\varphi+1|&:=(|\varphi_x+1|)_{x\in V},\\
		\sigma(\varphi)&:=({\rm sgn}(\varphi_x+1))_{x\in V},\\
		J(\varphi)_{xy}&:=|\varphi_{x}+1|\,|\varphi_y+1|.
		\end{split}
		\end{align}\paragraph{Step 1:}{\em  Conditionally on $|\varphi+1|$, the random variable $\sigma(\varphi)$ is distributed according to the Ising model on $\L$ with $+$ boundary conditions and coupling constants $J(\varphi)$.}
		\medbreak
		Let us mention that this is an observation that was already made in the literature (see e.g.~\cite{LupWer16}). Recall that the Ising model on $\L$ with $+$ boundary conditions and coupling constants $J=(J_{xy})$ is defined on configurations $\sigma=(\sigma_x:x\in \L)$ in $\{-1,+1\}^\L$ by 
		\begin{align*}
		\mu_{\L;J}^{+}(\sigma) &:= \frac{1}{\widetilde Z_{\Lambda}} \exp [-\mathcal  H_{\Lambda,J}(\sigma)], \end{align*}
		where $\widetilde Z_\L$ is a normalizing constant and $\mathcal H_{\L,J}(\sigma)$ is the {\em Hamiltonian} given by
		\begin{align*}
		\mathcal H_{\Lambda,J}(\sigma) &:= -\sum_{\substack{xy\in E\\ \{x,y\}\cap \L\ne \emptyset}} J_{xy}\,\sigma_x\sigma_y,
		\end{align*}
		where $\sigma$ is extended to $V$ by setting $\sigma_x=+1$ for every $x\in \L^c$.
		
		Now, the fact that $\varphi_x=0$ for every  $ x$ outside $\L$ obviously implies $\sigma(\varphi)_x=+1$. In addition, we have that
		\begin{align*}
		\mathcal{D}_{\Lambda}(\varphi) &=\tfrac{1}{2} \sum_{\substack{xy\in E\\ \{x,y\}\cap \L\ne \emptyset}}(\varphi_x-\varphi_y)^2 = F(|\varphi+1|) +\mathcal H_{\L,J(\varphi)}(\sigma(\varphi))
		,\end{align*}
		where $F$ is some function on $\R^\L$.
		This implies that 
		$$d\Pb_{\Lambda}[\varphi] = G(|\varphi+1|)\,\mu_{\lambda;J(\varphi)}^{+}(\sigma(\varphi)) d\varphi$$
		for all $\varphi\in\R^{\L}$, where $G$ is some function on $\R^\L$. Since $\varphi\mapsto (|\varphi+1|,\sigma(\varphi))$ is a bijection from $(\R\setminus\{-1\})^{\L}$ (which has total Lebesgue measure) to $\big(\R_{>0}\times\{-1,+1\}\big)^{\L}$, the above equation implies Step 1 readily.
		
		\paragraph{Step 2:} {\em For $S\subset \Lambda$, one has that $\Eb_{\L}[\P_{{\bf p}(\varphi)}(\nlr{}{}{S}{ \L^c})] \leq \Eb_{\L}[X_S^t(\varphi)].$}
		\medbreak
		This step relies on the Edwards-Sokal coupling (see \cite{Gri06} for details), which we recall for completeness. Sample a configuration $\sigma$ on $\Lambda$ according to the Ising model with $+$ boundary conditions and coupling constants $J_{xy}$. Then, construct a configuration $\omega$ on the edges in $E$ intersecting $\L$ as follows: for every edge $xy$, let $\omega_{xy}$ be a Bernoulli random variable with parameter $1-\exp(-2J_{xy}\mathbf 1_{\{\sigma_x=\sigma_y\}})$. Note that $\omega_{xy}=0$ automatically if $\sigma_x\ne\sigma_y$. Below, ${\rm P}_J$ denotes the law of $(\sigma,\omega)$ and ${\rm E}_J$ the expectation with respect to ${\rm P}_J$. (We only use this notation in this step.)
		
		The percolation process $\omega$ thus obtained is called the {\em random-cluster model} with cluster-weight $q=2$ and wired boundary condition, but this will be irrelevant for us. The important feature of this coupling will be that, conditionally on $\omega$, $\sigma$ is sampled as follows: 
		\begin{itemize}[noitemsep,nolistsep]
			\item every vertex connected to $\L^c$ receives the spin $+1$;
			\item for every connected component $\mathcal C$ of $\omega$ not intersecting $\L^c$, choose a spin $\sigma_\mathcal C$ equal to $+1$ or $-1$ with probability $1/2$, independently for each connected component, and set $\sigma_x=\sigma_{\mathcal C}$ for every $x\in\mathcal C$.\end{itemize}
		Given a realization of the GFF $\varphi$, we construct $\omega$ as above for $J(\varphi)$ and $\sigma(\varphi)$ as defined in \eqref{eq:def_J} (recall from Step 1 that $\sigma(\varphi)$ is indeed an Ising model with coupling constants $J(\varphi)$). As a direct consequence of the construction, conditionally on $\omega$, the law of the $\sigma_x$ for the vertices which are not connected to $ \L^c$ is symmetric by global flip. Applying these observations, we deduce that 
		\begin{align}\label{eq:a}
		\begin{split}
		\Eb_{\L}[X_S^t(\varphi)\,|\,|\varphi+1|] &\geq {\rm E}_{J(\varphi)}[{\rm E}_{J(\varphi)}(X_S^t(\varphi)|\omega) \,\mathbf{1}_{\{\nlr{}{}{S}{\L^c}~in~\omega\}}\,|\,|\varphi+1|]\\
		&\ge{\rm P}_{J(\varphi)}[\nlr{}{}{S}{\L^c}~in~\omega].
		\end{split}
		\end{align}
		In the last inequality we used that, conditionally on $|\varphi+1|$ and the event that $S$ is not connected to $ \L^c$ in $\omega$, $\log(X_S^t(\varphi))=\sum_{x\in S} t_x |\varphi_x+1|\sigma_x$ has mean $0$, so that ${\rm E}_{J(\varphi)}(X_S^t(\varphi)|\omega)\ge1$ by Jensen's inequality.
		
		Now, conditioned on $\sigma$, the only vertices that can potentially be connected to $ \L^c$ in $\omega$ are those which are connected by a path of pluses in $\sigma$. For an edge $xy$ with at least one endpoint of this type, one has $1-\exp(-2J_{xy}(\varphi)\mathbf 1_{\{\sigma(\varphi)_x=\sigma(\varphi)_y\}})={\bf  p}(\varphi)_{xy}$.
		This observation together with the Edwards-Sokal coupling implies
		\begin{equation*}{\rm P}_{J(\varphi)}[\nlr{}{}{S}{\L^c}~in~\omega\,|\,\sigma]
		={\bf P}_{{\bf p}(\varphi)}[\nlr{}{}{S}{\L^c}].\end{equation*}
		Integrating over $\sigma$ (given $|\varphi+1|$) and using Step 1 gives
		\begin{equation}\label{eq:b}{\rm P}_{J(\varphi)}[\nlr{}{}{S}{\L^c}~in~\omega]
		=\Eb_{\L}[{\bf P}_{{\bf p}(\varphi)}(\nlr{}{}{S}{\L^c})\,|\,|\varphi+1|].\end{equation}
		Step 2 follows readily by putting \eqref{eq:b} into \eqref{eq:a} and then integrating with respect to $|\varphi+1|$.
		\paragraph{Step 3:} {\em Passing to the infinite volume.}
		\medbreak
		Step 2 implies that for every $S\subset T\subset \L$,
		\begin{equation}\label{eq:c}\Eb_{\L}[\P_{{\bf p}(\varphi)}(\nlr{}{}{S}{T^c})]\le \Eb_{\L}[X_S^t(\varphi)].\end{equation}
		Since $S$ and $T$ are finite, the random variables considered in the previous inequality are continuous local observables of $\varphi$. Letting $\L$ tend to $V$, by weak convergence we obtain
		$$\Eb[\P_{{\bf p}(\varphi)}(\nlr{}{}{S}{T^c})]\le \Eb[X_S^t(\varphi)].$$
		Letting $T$ tend to $V$ concludes the proof.\end{proof}
	\begin{remark}
		Notice that in this section, we did not fully use the assumption that \eqref{eq:heat.kernel.bound} holds for $d>4$, from Theorem~\ref{thm:main}. The only property we needed from $G$ was its transience (implied by \eqref{eq:heat.kernel.bound} for $d>2$), so that we could consider the GFF in infinite volume.
	\end{remark}
	\medbreak
	\begin{remark}
		If we were only interested in proving $\Eb[\P_{{\bf p}(\varphi)}(\lr{}{}{x}{\infty})]>0$, but not the quantitative bound \eqref{eq.main.result}, we could have proceeded as follows. The Edwards-Sokal coupling implies that for any $x$
		$$  {\rm P}_{J(\varphi)}[\lr{}{}{x}{\L^c}~in~\omega] = {\rm E}_{J(\varphi)}[ \sigma_x].$$
		Using the above, in place of \eqref{eq:a}, and subsequently applying \eqref{eq:b} and integrating with respect to $| \varphi +1|$,  we deduce that
		$$ \Eb_\Lambda \big[ {\bf P}_{{\bf p}(\varphi)}(\lr{}{}{x}{\L^c})\big] = \Eb_\Lambda \big[{\rm sgn} (\varphi_x + 1)\big]. $$
		Proceeding as in the third step, we obtain 
		$ \Eb \big[ {\bf P}_{{\bf p}(\varphi)}(\lr{}{}{x}{\infty})\big] \geq \Eb \big[{\rm sgn} (\varphi_x + 1)\big] > 0$.\end{remark}

	\begin{remark}
	In this remark, we explain an alternative approach to Proposition~\ref{prop:gff.percolation} based on recent developments in the study of GFF on the cable system. Consider the (extended) GFF $\tilde{\varphi}$ on the metric graph $\tilde{G}$ constructed by interpreting each edge of $G$ as an interval where the field takes values continuously.
		By comparing with \cite{Lup16}, one can deduce that the clusters of the annealed percolation model with random parameters ${\bf p}(\varphi)$ exactly correspond to the connected components (when restricted to the vertices of $G$) of the super level-set $\{y\in \tilde{G}:~ \tilde{\varphi}_y + 1 > 0\}$. This connection follows from the following observations: first, the field $\tilde{\varphi}$ can be constructed from $\varphi$ by simply putting independent Brownian bridges on each edge, interpolating between the values on its endpoints; second, the probability that a Brownian bridge between $a$ and $b$ stays above $-1$ is exactly $1-\exp[-2(a+1)_+(b+1)_+]$ (see \cite{Lup16} for details).
Due to a theorem of Sznitman \cite{Szn12}, the sign-clusters of $\tilde{\varphi}_y+1$ are the same as the clusters of the superposition of a Poisson loop-soup of parameter $1/2$ with an independent random interlacement of intensity $1/2 $ on the metric graph $\tilde{G}$. The interlacement can be seen as ``infinite loops'' that arise from tilting the field by $+1$.
One can then prove that the super level-set $\{y\in \tilde{G}:~ \tilde{\varphi}_y + 1 > 0\}$ stochastically dominates a random interlacement of parameter $1/2$ (see for example Theorem 3 of \cite{Lup16}). Let us mention that the argument of Lupu is based on uniqueness of the infinite sign-cluster, which is currently written only for $\mathbb  Z^d$. Even though the uniqueness may fail for the graphs that we are studying, the domination mentioned above should still be true in our context. Taking this result as granted, and observing that the probability that $S$ intersects the random interlacement is equal to $1 - \exp[-\tfrac12\mathrm{cap}(S)]$, this would provide an alternative proof of Proposition \ref{prop:gff.percolation}. 
\end{remark}
	\medbreak
	\begin{remark} \label{rem:centered.field}
		In the same spirit as in the previous remark, Bernoulli percolation with random parameters given by ${\bf q}(\varphi)_{xy}:=1-\exp[-2(\varphi_x)_+(\varphi_y)_{+}]$ corresponds to the $0$ super level-set $\{y\in \tilde{G}:~ \tilde{\varphi}_y>0\}$. Also, using the strong Markov property for $\tilde{\varphi}$, Lupu proved in \cite{Lup16} that the sign of $\tilde{\varphi}$ can be sampled by assigning independent uniform signs to each excursion of $|\tilde{\varphi}|$. As a consequence, one has
		\begin{equation}\Eb[\P_{{\bf q}(\varphi)}(\lr{}{}{x}{y})] = \frac{1}{2}\Eb[{\rm sgn(\varphi_x)}{\rm sgn(\varphi_y)}]=\frac{1}{\pi}\arcsin \Big(\frac{\mathsf{G}(x,y)}{\sqrt{\mathsf{G}(x,x)\mathsf{G}(y,y)}}\Big)\end{equation}
		for every $x,y\in V$. One can easily deduce from this identity that the (annealed) percolation on the random environment ${\bf q}(\varphi)$ has infinite {\em susceptibility}, i.e. $\sum_{y\in V}\Eb[\P_{{\bf q}(\varphi)}(\lr{}{}{x}{y})]=+\infty$ for any $x\in V$.
	\end{remark}
	\medbreak
	\begin{remark}
		The previous remark shows that the two-point connectivity probabilities of the model with edge-parameters ${\bf q}(\varphi)$ tend to zero when the distance between the points diverges. This explains why we shift the GFF by 1:  the edge-parameters ${\bf p}(\varphi)$ are such that the two-point connectivity probabilities do not tend to zero, hence suggesting the existence of an infinite connected component.
	\end{remark}
	
	\section{Integrating out the random environment}
	
	If ${\bf p}(\varphi)$ was bounded away from 1, the result would follow by comparison between different Bernoulli percolations. Yet, the GFF is unbounded, and places where the field is large are places for which ${\bf p}(\varphi)$ is very close to 1, so that the vertices in these regions are almost automatically connected. As a consequence, we will need to consider the annealed probabilities.
	
	\begin{remark}Let us mention that we were very inspired by the beautiful paper of Rodriguez and Sznitman \cite{RodSzn13} on the study of the super level-set percolation of GFF.
	\end{remark}
	
	If we could prove that the annealed percolation on the random environment ${\bf p}(\varphi)$ was stochastically dominated by a Bernoulli percolation $\P_p$ with $p<1$, then we would be done. Unfortunately, this is not true (for example, one can prove that in $\Z^d$, the probability that all the edges inside a ball are open in the former decays slower than in the latter for any $p<1$). On the other hand, we are able to compare the probabilities {\em for ``connectivity events''} such as $\{\lr{}{}{S}{\infty}\}$.
	
	\begin{proposition}
		\label{prop:main}
		There exists $p<1$ such that for every finite subset $S$ of $V$,
		$$\P_{p}[\lr{}{}{S}{\infty}]\geq \Eb[\P_{{\bf p}(\varphi)}(\lr{}{}{S}{\infty})].$$
	\end{proposition}
	
	This proposition, together with Proposition~\ref{prop:gff.percolation}, implies Theorem~\ref{thm:main} readily. We therefore focus on the proof of the proposition.
	
	\begin{remark}
		It will be evident in the proof that we could also get the result of Proposition \ref{prop:main} for all events of the form $\{\lr{}{}{A}{B}\}$ where $A,B \subset V$ are finite. It is also clear that one could prove the same for ${\bf q}(\varphi)$ instead of ${\bf p}(\varphi)$, since ${\bf q}(\varphi)\leq {\bf p}(\varphi)$ (see Remark \ref{rem:centered.field}). This, together with Remark \ref{rem:centered.field}, would imply the existence of $p<1$ such that the susceptibility of Bernoulli percolation with parameter $p$ is infinite. Since for quasi-transitive graphs the susceptibility is finite in the whole subcritical phase (see \cite{AizBar87,Men86,DumTas15}), we would deduce that $p_c(G)\le p<1$. Therefore, if we only wanted to prove Theorem \ref{thm:main2}, it would be enough to consider the (perhaps more intuitive) random environment ${\bf q}(\varphi)$.
	\end{remark}
	
	The fact that the GFF has long-range dependencies is a difficulty here. In order to overcome this problem, the key tool used in the proof of Proposition \ref{prop:main} is a multi-scale decomposition of the GFF in terms of finite-range-dependent Gaussian fields. Such decompositions appear naturally in rigorous implementations of the Renormalization Group. In this context, the spin-spin correlations of a spin system (for instance the Ising model or the $\phi^4_d$ lattice models) with a certain set of parameters $\beta,\lambda,\dots$ are expressed in terms of the GFF $\varphi$, which itself is decomposed into a sum of fields with finite-range dependencies $\varphi=\sum_n\varphi^n$. Then, one {\em integrates out} the fields $\varphi^n$ one by one by changing the parameters $\beta,\lambda,\dots$ into parameters $\beta_1,\lambda_1,\dots$, then $\beta_2,\lambda_2,\dots$, etc. We will do the same in our context. The parameter that will vary in each step to compensate for the integration of the field $\varphi^n$ will be called $q_n$. A main difference with the Renormalization Group is that we will only be interested in inequalities; see \eqref{eq:induction.step} below.

	We now describe the decomposition that we are going to use in our proof. Let $q_n(x,y)$ be the heat kernel associated with the \textit{lazy} random walk in $ G$, i.e.~the Markov chain which stays put with probability $1/2$, and moves to one of the neighbors chosen uniformly at random with probability $1/2$. For any $x,y\in  V$, set $\mathsf{G}_0(x,y):=\tfrac{1}{2d(y)}q_0(x,y)$ and 
	$$\quad \mathsf{G}_n(x,y):=\frac{1}{2d(y)}\sum_{2^{n-1}\le k<2^n}q_k(x,y)$$
	for all $n\geq1$. The matrices $(\mathsf{G}_n)_n$ satisfy the following properties:
	\begin{enumerate}
		\item \label{decomp.G}
		$\mathsf{G}(x,y)=\sum_{n\geq0}\mathsf{G}_n(x,y)$ for all $x,y\in  V$,
		\item \label{cov.G}
		$\mathsf{G}_n$ is a covariance (i.e.~symmetric positive semi-definite) matrix  for every $n\geq0$,
		\item \label{posit.assoc}
		$\mathsf{G}_n(x,y)\geq0$ for any $x,y\in V$ and $n\geq0$,
		\item \label{range.G}
		$\mathsf{G}_n(x,y)=0$ for any $x,y\in  V$ with $d(x,y)\geq 2^n$,
		\item \label{bound.decomp}
		there exists $c'>0$ such that, for every $n\geq0$ and $x\in  V$, one has
		\begin{equation}
		\label{eq:bound.decomp}
		\mathsf{G}_n(x,x)\leq c'2^{-(\frac{d-2}{2})n}.
		\end{equation}
	\end{enumerate}
	Properties \ref{decomp.G}, \ref{posit.assoc} and \ref{range.G} are evident. Property \ref{cov.G} follows from the fact $q_n$ is positive semi-definite with respect to the reversible measure $\sum_{x\in V}d(x)\delta_x$ for every $n$ (this is why we take the lazy random walk instead of the simple one). Property \ref{bound.decomp} is a direct consequence of our assumption \eqref{eq:heat.kernel.bound} on the decay of the heat kernel $p_n$. Indeed, by decomposing into the total number of times the lazy random walk stays put, we deduce that $q_n(x,x)=2^{-n}\sum_{k=0}^{n} {n\choose k} p_k(x,x)$. Combining this identity with the bound \eqref{eq:heat.kernel.bound} on $p_n(x,x)$ one can easily prove that $q_n(x,x)$ also satisfies a bound like \eqref{eq:heat.kernel.bound} (with a possibly different constant $c$).
	
	It follows from Properties \ref{decomp.G} and \ref{cov.G} above that, if $\varphi^n\sim \mathcal{N}(0,\mathsf{G}_n)$ are \textit{independent} Gaussian fields, then 
	\begin{equation}
	\varphi=\sum_{n\geq0}\varphi^n
	\end{equation}
	in law (convergence of the series in $L^2$ and almost surely can be proved by the martingale convergence theorem, for example).
	Property \ref{range.G} is called the {\em finite-range} property (the value $2^n$ should be understood as the scale at which  correlations occur in $\varphi^n$). Property \ref{posit.assoc} implies that each field $\varphi^n$ is positively correlated and therefore, satisfies the FKG inequality, see \cite{Pit82}. Property \ref{bound.decomp}, which bounds the value of $\mathsf{G}_n(x,x)$, will be used to show that $\varphi^n$ is small. 
	
	\begin{remark}
		We will use the assumption $d>4$ only to guarantee that the exponent $\frac{d-2}{2}$ in the bound \eqref{eq:bound.decomp} is strictly larger than 1. Let us mention that in \cite{Bau13}, it was proved that there is a decomposition such that the bound \eqref{eq:bound.decomp} holds with exponent $d-2$ instead of $\frac{d-2}{2}$. Unfortunately, this decomposition does not seem to satisfy  Property \ref{posit.assoc}.
	\end{remark}
	
	From now on, we write $\Pb$ (resp. $\Eb$) for the probability (resp. expectation) with respect to $(\varphi^n)_{n\geq0}$, and set 
	$\varphi:=\sum\varphi^n$. By construction, $\varphi$ has the law of the GFF. For convenience (this will be clear in \eqref{eq:CS_ineq} below), we introduce the normalized Gaussian processes $\phi^n:=\tfrac{\pi (n+1)}{\sqrt{3}} \varphi^n$.
	
	For the proof, we add three copies of the edge $xy$ of $G$, that we denote $\overline{xy}$,  $\overrightharp{xy}$, $\overrightharp{yx}$ and call the new graph with all these edges $\overline G$ (it has the same set of vertices and four edges between every pair of neighbors in $G$). Despite the notation $\overrightharp{xy}$, we will regard $\overline G$ as an \textit{undirected} graph, so paths can go through $\overrightharp{xy}$ in both directions.
	We are going to make multiple uses of the following simple fact: the superposition (maximum) of two independent Bernoulli variables with parameters $1-e^{-a}$ and $1-e^{-b}$ is a Bernoulli variable with parameter $1-e^{-(a+b)}$. 
	Fix some $h\geq0$ to be determined below. For each realization of  $(\varphi^n)_{n \geq 0}$, define a Bernoulli percolation model ${\bf P}^h_{q,n,\lambda}$ on $\overline G$ with parameters \begin{align*}
	{\bf p}_{xy}&:=q,\\
	{\bf p}_{\overline{xy}}&:=1-\exp\big(-h-\sum_{k>n}(\phi^k_x)_+^2 + (\phi^k_y)_+^2\big),\\
	{\bf p}_{\overrightharp{xy}}&:=1-\exp\big(-(\phi^n_x)^2\mathbf 1_{\{ \phi^n_x\geq\lambda \} }\big).
	\end{align*}
	The edge-density of $\overline{xy}$ depends on the $\phi^k$ with $k>n$ only, those of $\overrightharp{xy}$ depend on $\phi^n$ only, and that of $xy$ is deterministic. Also, the parameter $\lambda$ enables us to interpolate between ${\bf P}^h_{q,n,0}$ and ${\bf P}^h_{q,n+1,0}$ (which corresponds to ${\bf P}_{q,n,\infty}$). 
	
	We now integrate out the randomness coming from the Gaussian processes by showing that there exists $h$ large enough and an increasing sequence $(q_n)$ such that $\lim_n q_n<1$ and
	$$\Eb[\P^h_{q_n,n,0}(\lr{}{}{S}{\infty})] \ge \Eb[\P^h_{{\bf p}(\varphi)}(\lr{}{}{S}{\infty})]$$
	for all $n$. We prove this by induction. The first lemma initiates the induction. 
	\begin{lemma}\label{lem:a1}
		For every $n_0\geq0$, there exists $h=h(n_0)>0$ such that 
		$$\Eb[\P^h_{0,n_0,0}(\lr{}{}{S}{\L^c})] \ge \Eb[\P_{{\bf p}(\varphi)}(\lr{}{}{S}{\L^c})]$$
		for every two finite subsets $S\subset \L$ of $V$.
	\end{lemma}
	\begin{proof}
		Using that $(1+a)(1+b)\le 2+a^2+b^2$, we find that
		\begin{align*}
		2(1+\varphi_x)_+\,(1+\varphi_y)_+
		&\leq 2\Big(1+\sum_{n\geq0}(\varphi^n_x)_+\Big) \Big(1+\sum_{n\geq0}(\varphi^n_y)_+\Big) \\
		&\leq 4+2\Big(\sum_{n\geq0}(\varphi^n_x)_+\Big)^2 + 2\Big(\sum_{n\geq0}(\varphi^n_y)_+\Big)^2.\end{align*}
		Using Cauchy-Schwarz twice together with the identity $\varphi^n_x=\tfrac{\sqrt{3}}{\pi (n+1)} \phi^n_x$ gives that
		\begin{align}\label{eq:CS_ineq}
		2(1+\varphi_x)_+\,(1+\varphi_y)_+&\leq 4+\sum_{n\geq0}\big[(\phi^n_x)_+^2 +(\phi^n_y)_+^2\big]
		\end{align}
		(here of course the positive part is taken before the square).
		Define $K_{xy}:=4+\sum_{k<n_0}\big[(\phi^k_x)_+^2+(\phi^k_y)_+^2\big]$ and ${\bf q}_{xy}:=1-\exp\big(-K_{xy}\big)$. It follows from the bound above that percolation with parameters ${\bf p}(\varphi)$ is stochastically dominated by the superposition of ${\bf P}^0_{0,n_0,0}$ and an independent percolation with parameter ${\bf q}$. Therefore we only need to show that there exists $h>0$ such that the annealed percolation model with (random) parameters ${\bf q}$ is stochastically dominated by a Bernoulli percolation with parameter $1-e^{-h}$. Notice that, for every $M>0$, this model is clearly dominated by the superposition of $\omega_{xy}:={\bf 1}_{\{ K_{xy}>M \}}$ and an independent Bernoulli percolation with parameter $1-e^{-M}$. As each $\phi^k$ has finite range of dependence, $\omega$ also does. Also notice that $G$ has uniformly bounded degree and $\Pb[\omega_{xy}=1]=\Pb[K_{xy}>0]$ does not depend on the edge $xy$. These observations together with the result \cite[Theorem 1.3]{LigSchSta97} implies that, provided that $M$ is chosen large enough (depending on $n_0$), $\omega$ is dominated by a Bernoulli percolation with parameter $1-e^{-1}$. Taking $h=M+1$ gives the result.
	\end{proof}
	The second lemma is used for the induction step. More precisely, it will allow us to remove continuously the field $\phi^n$ using a reasoning similar to the Aizenman-Grimmett paper \cite{AizGri} on essential enhancements. 
	\begin{lemma}\label{lem:a2}
		There exist $\alpha>0$ and $n_0\geq1$ depending only on $G$, such that for every two finite subsets $S\subset \L$ of $V$, every $h\geq0$, $n\ge n_0$, $\lambda\ge n^{-1}$ and $q\ge \tfrac12$, we have
		$$-\frac{\rm d}{{\rm d}\lambda}\Eb[\P^h_{q,n,\lambda}(\lr{}{}{S}{\L^c})]\le \exp\big(-\alpha 2^{n}\lambda^2 \big) \frac{\rm d}{{\rm d}q}\Eb[\P^h_{q,n,\lambda}(\lr{}{}{S}{\L^c})].$$
	\end{lemma}
	
	Before proving this lemma, let us show how Proposition~\ref{prop:main} follows from it.
	\begin{proof}[Proof of Proposition~\ref{prop:main}]
		Take $n_0$  and $h=h(n_0)> 0$ given by Lemmas \ref{lem:a2} and \ref{lem:a1} respectively. Define $q_n$ inductively by setting $q_{n}:=1/2$ for all $ n\le n_0$ and $q_{n+1}:=\lim_{\lambda\rightarrow\infty}q_n(\lambda)$ for all $n\ge n_0$, where for $\lambda\ge n^{-1}$,
		\begin{equation*}
		q_n(\lambda):= 
		q_n + 2n^{-2}+\int_{n^{-1}}^{\lambda}\exp(-\alpha 2^{n}t^2) dt .
		\end{equation*}
		First, notice that by possibly increasing $n_0$, we can guarantee that
		$$q:=\lim_{n\rightarrow\infty} q_n=\frac{1}{2}+\sum_{n\ge n_0}  \Big(2n^{-2}+\int_{n^{-1}}^{\infty}\exp(-\alpha 2^{n} t^2)dt\Big) < 1,$$ 
		so that every quantity written below makes sense. 
		
		Fix two finite subsets $S\subset\L$ of $V$. On the one hand, using that $$(\phi^n_x)^2\mathbf 1_{\{ \phi^n_x\geq0 \}}\leq n^{-2} + (\phi^n_x)^2\mathbf 1_{\{ \phi^n_x\geq n^{-1} \}},$$
		we obtain that 
		$$\Eb[\P^h_{q_n(1/n),n,1/n}(\lr{}{}{S}{\L^c})] \ge \Eb[\P^h_{q_n,n,0}(\lr{}{}{S}{\L^c})]$$
		(the $n^{-2}$ terms in ${\bf p}_{\overrightharp{xy}}$ and ${\bf p}_{\overrightharp{yx}}$ are ``transferred'' to ${\bf p}_{xy}$ by changing $q_n$ to $q_n(1/n)$; notice that this inequality is actually true even in the quenched sense).
		On the other hand, Lemma \ref{lem:a2} with the choice of $q_n(\lambda)$ tells us that the derivative of the function $$\lambda\mapsto \Eb[\P^h_{q_n(\lambda),n,\lambda}(\lr{}{}{S}{\L^c})]$$ is positive on $[n^{-1},\infty)$, so that the function is increasing in this interval. Taking $\lambda$ to infinity implies that 
		\begin{equation} \label{eq:induction.step}
		\Eb[\P^h_{q_{n+1},n+1,0}(\lr{}{}{S}{\L^c})]=\Eb[\P^h_{q_{n+1},n,\infty}(\lr{}{}{S}{\L^c})]\ge \Eb[\P^h_{q_{n},n,0}(\lr{}{}{S}{\L^c})]
		\end{equation} 
		for all $n\ge n_0$. This, together with Lemma \ref{lem:a1}, gives 
		$$\P_p[\lr{}{}{S}{\L^c}]=
		\lim_{n\rightarrow\infty}\Eb[\P^h_{q_{n},n,0}(\lr{}{}{S}{\L^c})]\ge \Eb[\P^h_{{\bf p}(\varphi)}(\lr{}{}{S}{\L^c})]$$
		where $p:=1-(1-q)e^{-h}$. The result follows by letting $\L$ tend to $V$.
	\end{proof}
	
	We now go back to the proof of Lemma~\ref{lem:a2}.
	Let us first recall classical expressions for derivatives of events in Bernoulli percolation.
	Consider an increasing event $A$ depending on finitely many edges. A {\em set} $F$ of edges in $\overline G$ is {\em pivotal} (in $\omega$) for $A$ if the configuration is in $A$ when one opens all the edges in $F$ and not in $A$ when one closes these edges. We say that $F$ is {\em open} (resp.~{\em closed}) {\em pivotal} if in addition $\omega\in A$ (resp.~$\omega\notin A$). Notice that $F$ being open pivotal {\em does not} necessarily imply that {\em all} the edges in $F$ are open. Of course, all these definitions apply when $F$ consists of a single edge  to recover the standard notion of pivotality.
	Russo's formula states that 
	\begin{equation}\label{eq:d1}
	\frac{\rm d}{{\rm d}q}\Eb[\P^h_{q,n,\lambda}(A)]=\sum_{xy\in E}\Eb[\P^h_{q,n,\lambda}(xy\text{ pivotal for }A)].
	\end{equation}
	For the derivative in $\lambda$, a quick analysis of $\tfrac1\delta\Eb[\P^h_{q,n,\lambda+\delta}(A)-\P^h_{q,n,\lambda}(A)]$ gives that 
	\begin{equation}\label{eq:d2}
	-\frac{\rm d}{{\rm d}\lambda}\Eb[\P^h_{q,n,\lambda}(A)]=\sum_{x} \rho^n_x(\lambda) \Eb[\P^h_{q,n,\lambda}(N_x\text{ open pivotal for }A) | \phi^n_x=\lambda],
	\end{equation}
	where $\rho^n_x(\cdot)$ is the density of $\phi^n_x$ and $N_x:=\{\overrightharp{xy}:~ xy\in E\}$ is the directed edge neighborhood of $x$.
	
	\begin{proof}[Proof of Lemma~\ref{lem:a2}]
		Fix any $h\geq 0$ and $n\geq 1$. To lighten the notation, write $L=2^n$ and $\P_{q,\lambda}$ instead of $\P^h_{q,n,\lambda}$ and keep in mind that $\P_{q,\lambda}$ is a function of $(\phi^k)_{k\ge n}$.  Below, we apply the notions defined in the last paragraphs for $A$ being equal to $\{\lr{}{}{S}{\L^c}\}$, where $\L$ is a finite set of vertices containing $S$. In order to lighten the notation, we write ``pivotal'' instead of ``pivotal for $\{\lr{}{}{S}{\L^c}\}$''.

		The proof proceeds as follows. We start from the quantity obtained on the right of \eqref{eq:d2}, and try to compare it to the one obtained in \eqref{eq:d1}. We do it in three steps. The first one consists in going from open to closed pivotal. The second one enables us to get rid of the conditioning on $\phi^n_x=\lambda$, at the cost of comparing to the probability that the ball $B_L(x)$ of radius $L$ around $x$ is pivotal. The third step brings us back from the probability of the latter to probabilities of being pivotal for individual edges.
		
		\paragraph{Step 1.} {\em From $x$ open pivotal to $x$ closed pivotal.}
		\medbreak
		Since $N_x$ being pivotal is independent of the state at $N_x$, we deduce that 
		\begin{align}\nonumber
		\Eb[\P_{q,\lambda}(N_x\text{ closed pivotal}) | \phi^n_x=\lambda]&\ge \Eb\big[\P_{q,\lambda}(N_x\text{ closed})\cdot\P_{q,\lambda}(N_x\text{ pivotal}) \big| \phi^n_x=\lambda\big]\\
		& = \exp\big(-d(x)\lambda^2\big)\,\Eb[\P_{q,\lambda}(N_x\text{ pivotal}) | \phi^n_x=\lambda]\nonumber\\
		& \ge \exp\big(-d(x)\lambda^2\big)\,\Eb[\P_{q,\lambda}(N_x\text{ open pivotal}) | \phi^n_x=\lambda].\label{eq:c2}
		\end{align}
		\paragraph{Step 2.} {\em Removing the conditioning on $\phi^n_x=\lambda$.}
		\medbreak
		For $N_x$ to be closed pivotal, the ball $B_L(x)$ must be closed pivotal. We therefore find that 
		$$\Eb[\P_{q,\lambda}(N_x\text{ closed pivotal})|\phi^n_x=\lambda]\le \Eb[\P_{q,\lambda}(B_L(x)\text{ closed pivotal})|\phi^n_x=\lambda].$$
		Call $\tilde{\mathsf{G}}_n$ the covariance matrix of $\phi^n$. Conditionally on $\phi^n_x=\lambda$, $\phi^n$ is a Gaussian process with mean and covariance given, respectively, by 
		$$m_z:=\lambda\frac{\tilde{\mathsf{G}}_n(x,z)}{\tilde{\mathsf{G}}_n(x,x)}\quad\text{and}\quad \mathsf{G}'_n(z,w):=\tilde{\mathsf{G}}_n(z,w)-\frac{\tilde{\mathsf{G}}_n(z,x)\tilde{\mathsf{G}}_n(x,w)}{\tilde{\mathsf{G}}_n(x,x)}$$
		for every $z,w\in V$.  In particular, for every $\mu\le\lambda$, $\phi^n$ conditioned on $\phi^n_x=\lambda$ and $\phi^n$ conditioned on $\phi^n_x=\mu$ are shifts of the same centered Gaussian process. Since the shift $(\lambda-\mu)\tilde{\mathsf{G}}_n(x,z)/\tilde{\mathsf{G}}_n(x,x)$ is non-negative for $z\in B_L(x)$ and equal to 0 for $z\notin B_L(x)$ (by Properties \ref{posit.assoc} and \ref{range.G} of $(\mathsf{G}_n)$, respectively), we deduce that 
		$$\Eb[\P_{q,\lambda}(B_L(x)\text{ closed pivotal})|\phi^n_x=\lambda]\le \Eb[\P_{q,\lambda}(B_L(x)\text{ closed pivotal})|\phi^n_x=\mu].$$
		Multiplying by $\rho_x^n(\mu)$ and integrating on $\mu\le \lambda$ gives that \begin{equation*}\Eb[\P_{q,\lambda}(B_L(x)\text{ closed pivotal})|\phi^n_x=\lambda]\le \frac{\Eb[\P_{q,\lambda}(B_L(x)\text{ closed pivotal})]}{\Pb[\phi^n_x\le \lambda]}.\end{equation*}
		Using $\Pb[\phi^n_x\le \lambda]\ge \tfrac12$ and \eqref{eq:c2} gives that 
		\begin{equation}\label{eq:b1}\Eb[\P_{q,\lambda}(N_x\text{ open pivotal}) | \phi^n_x=\lambda]\le 2 \exp(d(x)\lambda^2) \ \Eb[\P_{q,\lambda}(B_L(x)\text{ closed pivotal})].\end{equation}
		\paragraph{Step 3.} 
		{\em From a pivotal ball to a pivotal edge.
		} 
		\medbreak
		Fix an order for vertices and edges of $G$ and consider a configuration $\omega$ in $\Lambda$ for which $B_L(x)$ is closed pivotal. Let $y$ and $z$ be the smallest vertices in $B_L(x)$ such that $\lr{}{}{S}{y}$ and $\lr{}{}{z}{\L^c}$ both without using any edge contained in $B_L(x)$. Take $\gamma$ in $G$ to be the earliest (in lexicographical order) path contained in $B_L(x)$ of length at most $2L$ between $y$ and $z$, and define a configuration $\omega'$ by opening the edges of $\gamma$ one by one (in order) until the first time that an edge $uv$ of $B_L(x)$ becomes pivotal. Recall that every edge of $G$ (in particular the ones in $\gamma$, which we opened) have parameter $q$ under $\P_{q,\lambda}$.
		By construction, $\omega'$ contains a pivotal edge in $B_L(x)$, and it is elementary to check that 
		$$\Eb[\P_{q,\lambda}(\omega')]\ge q^{2L} \,\Eb[\P_{q,\lambda}(\omega)].$$ Furthermore, the map from $\omega$ to $\omega'$ is at most $2^{2L}$-to-one (since the configuration is not altered outside $B_L(x)$, the sites $y$ and $z$ can be reconstructed, and so can $\gamma$). We deduce that 
		\begin{equation}\label{eq:b5}\Eb[\P_{q,\lambda}(B_L(x)\text{ closed pivotal})]\le 2^{2L} q^{-2L}\sum_{\substack{uv\in E:\\u,v\in B_L(x)}}\Eb[\P_{q,\lambda}(uv\text{ pivotal})].\end{equation}
		\paragraph{Conclusion of the proof.} Let $D:=\max_x d(x)$ be the maximum degree of $G$. Combining the two inequalities \eqref{eq:b1} and \eqref{eq:b5} enables us to compare the right-hand sides of \eqref{eq:d2} and \eqref{eq:d1}: 
		\begin{align*}
		-\frac{\rm d}{{\rm d}\lambda}\Eb[\P_{q,\lambda}(\lr{}{}{S}{\L^c})]
		&\le \big(\sup_{x\in V} \rho^n_x(\lambda)\big) \cdot \Big(\frac{4D}{q^2}\Big)^L\exp(D\lambda^2) \sum_{uv\in E} \Eb[\P_{q,\lambda}(uv\text{ pivotal})]\\
		&\le  \big(\sup_{x\in V} \rho^n_x(\lambda)\big) \cdot \exp(C 2^n + D\lambda^2) \frac{\rm d}{{\rm d}q}\Eb[\P_{q,\lambda}(\lr{}{}{S}{\L^c})]
		\end{align*}
		for some constant $C>0$. We replaced $L$ by $2^n$ and used that $q\ge 1/2$ and that the number of times that an edge $uv$ appears in the summation is $|\{x\in V: u,v\in B_L(x)\}|\leq D^{L}$. Recalling that $\rho^n_x(\lambda):=\tfrac{1}{\sqrt{2\pi \tilde{\mathsf{G}}_n(x,x)}}\exp[-\tfrac12\lambda^2/\tilde{\mathsf{G}}_n(x,x)]$, where $\tilde{\mathsf{G}}_n$ is the covariance matrix of $\phi^n$, and using that $$\tilde{\mathsf{G}}_n(x,x)=\frac{\pi^2 (n+1)^2}{3}\mathsf{G}_n(x,x)\stackrel{\eqref{eq:bound.decomp}}\le c'' 2^{-\beta n}$$ for some $c''>0$ and $\beta>1$ (here is the only place where we use $d>4$), one can find $n_0\geq0$ and $\alpha>0$ such that 
		$$\big(\sup_{x\in V} \rho^n_x(\lambda)\big) \cdot \exp\big(C 2^n + D\lambda^2\big)\le \exp[-\alpha 2^{\beta n}\lambda^2]$$
		for every $n\geq n_0$, $\lambda\ge n^{-1}$ and $q\ge1/2$, thus concluding the proof.
	\end{proof}
	
	\paragraph{Acknowledgments} This project was initiated during a visit to IHES. The authors thank the institution for making this collaboration possible. This research was supported by the ERC CriBLaM, an IDEX grant from Paris-Saclay, and the NCCR SwissMAP. AY is partially supported by the Israel Science Foundation (grant no. 1346/15). Severo's work was also 
	supported by Fonds National Suisse. We thank Roland Bauerschmidt, Itai Benjamini, Tom Hutchcroft, S\'ebastien Martineau and Alain-Sol Sznitman for comments and references. We would also like to thank the anonymous referees for their valuable comments and corrections. 
	
	\providecommand{\bysame}{\leavevmode\hbox to3em{\hrulefill}\thinspace}
	\providecommand{\MR}{\relax\ifhmode\unskip\space\fi MR }
	\providecommand{\MRhref}[2]{%
		\href{http://www.ams.org/mathscinet-getitem?mr=#1}{#2}
	}
	\providecommand{\href}[2]{#2}


\begin{thebibliography}{BLPS99}
		
		
		\bibitem[AB87]{AizBar87}
		M.~Aizenman and D.~Barsky, \emph{Sharpness of the phase transition in
			percolation models}, Communications in Mathematical Physics \textbf{108} (1987), no.~3, 489--526.
		
		\bibitem[AG91]{AizGri}
		M.~Aizenman and G.~Grimmett, \emph{Strict monotonicity for critical points in percolation and ferromagnetic models},
		Journal of Statistical Physics {\bf 63} (1991), 817--835.
		
		\bibitem[BB99]{babson1999cut}
		E.~Babson and I.~Benjamini, \emph{Cut sets and normed cohomology with
			applications to percolation}, Proceedings of the American Mathematical
		Society \textbf{127} (1999), no.~2, 589--597.
		
		\bibitem[Bal99]{Bal}
		T. Balaban and M.~O'Carroll, {\em Low
			temperature properties for correlation functions in classical $n$-vector
			spin models,}
		Communications in Mathematical Physics  \textbf{199}  (1999),  no.~3, 493--520.
		
		\bibitem[BB07]{balister2007counting}
		P.~Balister and B.~Bollob{\'a}s, \emph{Counting regions with bounded surface
			area}, Communications in Mathematical Physics \textbf{273} (2007), no.~2,
		305--315.
		
		\bibitem[Bau13]{Bau13}
		R.~Bauerschmidt, \emph{A simple method for finite range decomposition of
			quadratic forms and {G}aussian fields}, Probability Theory and Related Fields
		\textbf{157} (2013), no.~3, 817--845.
		
		\bibitem[BLPS99]{benjamini1999group}
		I.~Benjamini, R.~Lyons, Y.~Peres, and O.~Schramm, \emph{Group-invariant
			percolation on graphs}, Geometric \& Functional Analysis GAFA \textbf{9}
		(1999), no.~1, 29--66.
		
		\bibitem[BPP98]{BPP}
		I.~Benjamini, R.~Pemantle, and Y.~Peres, \emph{Unpredictable paths and
			percolation}, The Annals of Probability \textbf{26} (1998), no.~3,
		1198--1211.
		
		\bibitem[BS96]{BenSch96}
		I.~Benjamini and O.~Schramm, \emph{Percolation beyond {$\bf Z^d$}, many
			questions and a few answers}, Electronic Communications in Probability \textbf{1} (1996), no.~8, 71--82 (electronic). 
		
		\bibitem[Ber16]{Ber16}
		N.~Berestycki, \emph{Introduction to the Gaussian free field and Liouville quantum gravity}, Available
		at http://www.statslab.cam.ac.uk/~beresty/Articles/oxford (2016)
		
		\bibitem[CT15]{candellero2015percolation}
		E.~Candellero and A.~Teixeira, \emph{Percolation and isoperimetry on roughly
			transitive graphs}, arXiv:1507.07765 (2015).
		
		\bibitem[CSC93]{CSC93}
		T.~Coulhon and L.~Saloff-Coste, \emph{Isop\'erim\'etrie pour les groupes et les vari\'et\'es}, 
		Revista Matem\'atica Iberoamericana \textbf{9} (1993), 293--314.
		
		\bibitem[DGRS19]{DGRS19}
		H.~{Duminil-Copin}, S.~Goswami, P.-F.~Rodriguez, and F.~Severo \emph{Equality of critical parameters for percolation of Gaussian free field level-sets}, arXiv:2002.07735 (2019).
		
		\bibitem[DT16]{DumTas15}
		H.~{Duminil-Copin} and V.~Tassion, \emph{A new proof of the sharpness of the
			phase transition for {B}ernoulli percolation and the {I}sing model},
		Communications in {M}athematical {P}hysics \textbf{343} (2016), no.~2,
		725--745.
		
		\bibitem[FSS76]{FSS76}
		J.~Fr\"ohlich, B.~Simon, and T.~Spencer,
		\newblock {\em Infrared bounds, phase transitions and continuous symmetry breaking},
		\newblock Communications in {M}athematical {P}hysics, \textbf{50} (1976), no.~1, 79--95.
		
		\bibitem[FS82]{FS82}
		J.~Fr\"ohlich and T.~Spencer,
		\newblock {\em Massless phases and symmetry restoration in abelian gauge
			theories and spin systems},
		\newblock Communications in {M}athematical {P}hysics,
		{\bf 83} (1982), no.~3, 411-454.
		
		\bibitem[Gri06]{Gri06}
		G.~R.~Grimmett, \emph{The random-cluster model}, Grundlehren der Mathematischen
		Wissenschaften [Fundamental Principles of Mathematical Sciences], vol. 333,
		Springer-Verlag, Berlin, 2006. 
		
		\bibitem[Gro81]{gromov1981groups}
		M.~Gromov, \emph{Groups of polynomial growth and expanding maps}, Publications
		Math{\'e}matiques de l'Institut des Hautes {\'E}tudes Scientifiques
		\textbf{53} (1981), no.~1, 53--78.
		
		\bibitem[GS98]{GriSta98}
		G.~R.~Grimmett and A.~M.~Stacey, \emph{Critical probabilities for site and bond
			percolation models}, The Annals of Probability \textbf{26} (1998), no.~4, 1788--1812.
		
		\bibitem[KK86]{KK86}
		T.~Kennedy and C.~King, 
		{\em Spontaneous symmetry breakdown in the abelian {H}iggs model},
		Communications in {M}athematical {P}hysics, {\bf 104} (1986), no.~2, 327--347.
		
		\bibitem[KZ90]{KZ90}
		H.~Kesten and Y.~Zhang, \emph{The probability of a large finite cluster in supercritical {B}ernoulli percolation}, The Annals of Probability \textbf{18} (1990), no.~2, 537--555.
		
		\bibitem[Koz05]{kozma2005percolation}
		G.~Kozma, \emph{Percolation, perimetry, planarity}, Revista Matem\'atica
		Iberoamericana \textbf{23} (2005) no.~2, 671--676.
		
		\bibitem[Law91]{Law91}
		G.~F.~Lawler, \emph{Intersections of random walks}, Birkh{\"a}user Boston Inc., 1991.
		
		\bibitem[LM98]{lebowitz1998improved}
		J.~Lebowitz and A.~Mazel, \emph{Improved {P}eierls argument for
			high-dimensional {I}sing models}, Journal of statistical physics \textbf{90}
		(1998), no.~3--4, 1051--1059.
		
		\bibitem[LMS08]{LyoMorSch08}
		R.~Lyons, B.~Morris, and O.~Schramm, \emph{Ends in uniform
			spanning forests},Electronic Journal of Probability \textbf{13} (2008), 1702--1725.
		
		\bibitem[LP17]{LyoPer17}
		R.~Lyons and Y.~Peres, \emph{Probability on trees and networks},
		Cambridge Series in Statistical and Probabilistic Mathematics, Cambridge
		University Press, 2017.
		
		\bibitem[LSS97]{LigSchSta97}
		T.~M.~Liggett, R.~H.~Schonmann, and A.~M.~Stacey, \emph{Domination by product
			measures}, The Annals of Probability \textbf{25} (1997), no.~1, 71--95.
		
		\bibitem[Lup16]{Lup16}
		T.~Lupu, \emph{From loop clusters and random interlacements to the free
			field}, The Annals of Probability \textbf{44} (2016), no.~3, 2117--2146.
		
		\bibitem[LW16]{LupWer16}
		T.~Lupu and W.~Werner, \emph{A note on {I}sing random currents,
			{I}sing-{FK}, loop-soups and the {G}aussian free field}, Electronic Communications in Probability \textbf{21} (2016), 7 pp.
		
		\bibitem[Lyo95]{Lyons95}
		R.~Lyons, \emph{Random walks and the growth of groups}, Comptes rendus de
		l'Acad{\'e}mie des sciences. S{\'e}rie 1, Math{\'e}matique \textbf{320}
		(1995), no.~11, 1361--1366.
		
		\bibitem[Men86]{Men86}
		M.~V.~Menshikov, \emph{Coincidence of critical points in percolation problems},
		Doklady Akademii Nauk SSSR \textbf{288} (1986), no.~6, 1308--1311.
		
		\bibitem[MP05]{MorPer05}
		B.~Morris and Y.~Peres, \emph{Evolving sets, mixing and heat kernel bounds},
		Probability Theory and Related Fields \textbf{133} (2005), no.~2, 245--266.
		
		\bibitem[MP01]{MP01}	
		R.~Muchnik and I.~Pak, \emph{Percolation on Grigorchuk groups},
		Communications in Algebra \textbf{29} (2001), 661--671.
		
		\bibitem[Nek18]{Nek}
		V.~Nekrashevych, \emph{Palindromic subshifts and simple periodic groups of intermediate growth},
Annals of Mathematics, \textbf{187(3)}, 2018, 667--719. 

		\bibitem[Pei36]{Pei36}
		R.~Peierls, \emph{On {I}sing's model of ferromagnetism.}, Mathematical Proceedings of the Cambridge Philosophical Society \textbf{32} (1936), 477--481.
		
		\bibitem[Pit82]{Pit82}
		L.~D.~Pitt, \emph{Positively correlated normal variables are associated.}, The Annals of Probability \textbf{10} (1982), no.~2, 496--499.
		
		\bibitem[RY17]{raoufi2017indicable}
		A.~Raoufi and A.~Yadin, \emph{Indicable groups and $p_c< 1$}, Electronic
		Communications in Probability \textbf{22} (2017).
		
		\bibitem[RS13]{RodSzn13}
		P.-F.~Rodriguez and A.-S.~Sznitman, {\em Phase Transition and Level-Set Percolation for the Gaussian Free Field}, Communications in Mathematical Physics {\bf 320} (2013), no.~2, 571--601.
		
		\bibitem[Szn12]{Szn12}
		A.-S. Sznitman, { \em An isomorphism theorem for random interlacements.}, Electronic Communications in Probability {\bf 17} (2012).
		
		\bibitem[Tei16]{teixeira2016percolation}
		A.~Teixeira, \emph{Percolation and local isoperimetric inequalities},
		Probability Theory and Related Fields \textbf{165} (2016), no.~3--4, 963--984.
		
		\bibitem[Tim07]{timar2007cutsets}
		{\'A}.~Tim{\'a}r, \emph{Cutsets in infinite graphs}, Combinatorics, Probability
		and Computing \textbf{16} (2007), no.~1, 159--166.
		
		\bibitem[Tro84]{Tro84}
		V.~I.~Trofimov, \emph{Graphs with polynomial growth}, Matematicheskii Sbornik (N.S.)
		\textbf{123(165)} (1984), no.~3, 407--421.
		
		\bibitem[Var85]{Var85}
		N.~Varopoulos, \emph{Isoperimetric inequalities and markov chains}, Journal of
		Functional Analysis \textbf{63} (1985), no.~2, 215--239.
		
		\bibitem[Woe00]{woess2000random}
		W.~Woess, {\em Random walks on infinite graphs and groups}, volume 138, Cambridge university press, 2000.
		
	\end{thebibliography}
\end{document}